\documentclass[12pt,reqno]{amsart}
\usepackage{amsmath,amssymb,amsfonts,amscd,latexsym,amsthm,mathrsfs,verbatim,comment}
\usepackage{enumerate}
\usepackage[unicode]{hyperref}
\newtheorem{theorem}{Theorem}[section]
\newtheorem{prop}[theorem]{Proposition}
\theoremstyle{remark}
\newtheorem{remark}[theorem]{\bf Remark}
\newtheorem{example}[theorem]{\bf Example}
\newtheorem*{acknowledgements}{\bf Acknowledgements}
\newcommand{\ba}{\boldsymbol a}
\newcommand{\bt}{\boldsymbol t}
\newcommand{\bs}{\boldsymbol s}
\newcommand{\bF}{\boldsymbol F}
\newcommand{\N}{\bold N}
\newcommand{\cA}{\mathcal A}
\newcommand{\cP}{\mathcal P}
\newcommand\Gal{\operatorname{Gal}}
\newcommand\Frob{\operatorname{Frob}}
\newcommand\ord{\operatorname{ord}}
\begin{document}

\title[Irrationality and transcendence in the `poor man's ad\`ele ring']{Irrationality and transcendence questions\\in the `poor man's ad\`ele ring'}

\author{Florian Luca}
\address{Stellenbosch University, Mathematics, Merriman Street 7600 Stellenbosch, South Africa}
\urladdr{https://florianluca.com/}

\author{Wadim Zudilin}
\address{Department of Mathematics, IMAPP, Radboud University, PO Box 9010, 6500~GL Nij\-me\-gen, The Netherlands}
\urladdr{https://www.math.ru.nl/~wzudilin/}

\date{18 May 2025}

\dedicatory{To Pieter Moree: Gefeliciteerd met je 60ste verjaarjaar!}


\maketitle

\begin{abstract}
We discuss arithmetic questions related to the `poor man's ad\`ele ring' $\mathcal A$ whose elements are encoded by sequences $(t_p)_p$ indexed by prime numbers, with each $t_p$ viewed as a residue in $\mathbb Z/p\mathbb Z$.
Our main theorem is about the $\mathcal A$-transcendence of the element $(F_p(q))_p$, where $F_n(q)$ (Schur's $q$-Fibonacci numbers) are the $(1,1)$-entries of $2\times2$-matrices
$$
\begin{pmatrix} 1 & 1 \\ 1 & 0 \end{pmatrix}
\begin{pmatrix} 1 & 1 \\ q & 0 \end{pmatrix}
\begin{pmatrix} 1 & 1 \\ q^2 & 0 \end{pmatrix}
\cdots
\begin{pmatrix} 1 & 1 \\ q^{n-2} & 0 \end{pmatrix}
$$
and $q>1$ is an integer. This result was previously known for $q>1$ square free under the GRH.
\end{abstract}

%==================================================

\section{Motivational introduction}
\label{s1}

An intriguing version of finite multiple zeta values (FMZVs) introduced in \cite{KZ25} (see also \cite{Ka19,KMS25}) assumes values in the `poor man's ad\`ele ring'
\[
\cA=\bigg(\prod_p(\mathbb Z/p\mathbb Z)\bigg)\bigg/\bigg(\bigoplus_p(\mathbb Z/p\mathbb Z)\bigg),
\]
which is the quotient ring of the direct product of $\mathbb Z/p\mathbb Z$ taken over all primes $p\in\cP$ modulo the ideal of the direct sum. The ring was introduced by Kontsevich in~\cite{Ko09} for somewhat different purposes. The field $\mathbb Q$ of rational numbers is naturally embedded in $\cA$ diagonally, so that $\cA$ is regarded as a $\mathbb Q$-algebra via this embedding.
We can encode elements $\bt$ of $\cA$ as infinite vectors $\bt=(t_2,t_3,t_5,t_7,\dots)=(t_p)_{p\in\cP}$, where $t_p\in\mathbb Q$, with the equivalence relation $\bt\sim\bs$ corresponding to $t_p\equiv s_p\bmod p$ for all primes $p\gg1$ (in particular, both $t_p\bmod p$ and $s_p\bmod p$ are well defined for such primes~$p$); a few first components of $\bt$ can be even undefined (and denoted as $\infty$ for convenience).
In particular, being rational in $\cA$ simply means being equivalent to a constant element $\bt=(t_p)_{p\in\cP}$ with $t_p=t\in\mathbb Q$ for all $p\gg1$.

One of the challenges posed by Kaneko and Zagier in \cite{KZ25} is to show that the model of FMZVs is non-trivial; for example, that the elements $Z_{\cA}(k)=(B_{p-k}/k)_{p\in\cP}\in\cA$ are \emph{non-zero} for $k>1$ odd.
This seems to be equally hard to demonstrating the \emph{irrationality} of $Z_{\cA}(k)$, equivalently, of $(B_{p-k}\bmod p)_{p\in\cP}\in\cA$.

As a simple illustration of the irrationality concept in the poor-man's-ad\`ele-ring setting, we give an example which shares similarity with the ancient proof of the irrationality of $\sqrt{2}$ over~$\mathbb Q$.

\begin{example}
\label{ex:sqrt2}
Truncations of the series for
\[
\sqrt{1+x}=1+\frac12x+\dotsb=1+\sum_{k=1}^\infty\frac{(-1)^{k-1}C_{k-1}}{2^{2k-1}}\,x^k,
\quad\text{where}\; C_k=\binom{2k}{k}-\binom{2k}{k+1},
\]
at $x^{p-1}$ start
\begin{gather*}
s_2(x)=1 + \frac12x, \quad
s_3(x)=1 + \frac12x - \frac18x^2, \quad
s_5(x)=1 + \frac12x - \frac18x^2 + \frac1{16}x^3 - \frac5{128}x^4, \\
s_7(x)=1 + \frac12x - \frac18x^2 + \frac1{16}x^3 - \frac5{128}x^4 + \frac7{256}x^5 - \frac{21}{1024}x^6, \quad\dotsc.
\end{gather*}
Their modulo $p$ residues at $x=1$ read
\[
\bs(1)=\big(s_p(1)\bmod p\big)_p
=(\infty,-2, -2, 2, -2,\dots)
=2\big(\infty,(\tfrac{2}{p})\big)_p,
\]
where $(\tfrac{\,\cdot\,}{\,\cdot\,})$ denotes the Legendre--Kronecker symbol.
Assume there is a rational $r$ such that $r\equiv(\tfrac{2}{p})\bmod p$ for all primes $p\gg1$; then $r^2\equiv1\bmod p$ for all such $p$. Since $r^2$ is also rational and $(r^2)_p\sim(1)_p$, we get $r^2=1$. This means that either $r=-1$ or $r=1$, and any of the options contradict $r\equiv(\tfrac{2}{p})\bmod p$, since the Legendre symbol assumes both values $1$ and $-1$ infinitely often. The contradiction obtained means that $\bs(1)$\,---\,which is nothing but a version of $\sqrt2$ in the ring\,---\,is irrational.
\end{example}

There is a general construction of algebraic numbers in the ring $\cA$ (see \cite{Ro20,RTTY24}), which we review in Section~\ref{s2} below.
Roughly speaking, one starts with a $C$-finite sequence $(t_m)_{m\gg1}\in\mathbb Q^\infty$, that is, a sequence satisfying a linear recursion with \emph{constant} rational coefficients, and associate to it $\bt=(t_p\bmod p)_{p\in\cP}\in\cA$ (with all $t_p$ for $p<N$ set to be~$\infty$, say).
Then $\bt$ is said to be an algebraic number in $\cA$; one can further give a notion of degree based on the degree of minimal irreducible characteristic polynomial of the recursion.
For example, rational numbers come out of the simple recurrence equation $t_{n+1}-t_n=0$.

With this setup one can naturally start asking questions about transcendence of elements in $\cA$.
The following transcendence criterion from \cite{AF24} is almost immediate; it follows from Theorem~\ref{thm:Ro} we state below.

\begin{prop}[{\cite[Proposition 3.7]{AF24}}]
\label{pr1}
Let $\alpha\in\cA$. Assume that there exists a sequence of \emph{distinct} integers $(a_n)_{n}$ such that $a_n$ occurs infinitely often in $\alpha$ for every $n$ in ${\mathcal N}$, then $\alpha$ is an ${\mathcal A}$-transcendental number. 
\end{prop}  

Following Schur, define the $q$-Fibonacci numbers $(F_n(q))_{n\ge0}$ via $F_0(q)=0$, $F_1(q)=1$ and
\[
F_{n+2}(q)=F_{n+1}(q)+q^nF_n(q) \quad \text{for}\; n\ge 0;
\]
when $q=1$ this coincides with the Fibonacci sequence.
The criterion above is applied in \cite[Theorem~1.1]{AF24} to prove the $\cA$-transcendence of the element $\bF(q)=(F_p(q))_{p\in\cP}\in\cA$ in the case when integer $q>1$~is square free, assuming the Generalized Riemann Hypothesis.
Our principal result demonstrates that these heavy restrictions can be dropped.

\begin{theorem}
\label{thm:Ourtheorem}
Let $q\in\mathbb Z_{>1}$. The element $\bF(q)=(F_p(q))_{p\in\cP}$ is $\cA$-transcendental. 
\end{theorem}

Theorem \ref{thm:Ourtheorem} is best possible (for positive integers~$q$) since $F_n(q)$ is constant $1$ for $q=0$ and $n\ge 1$, while it is the Fibonacci sequence for $q=1$. Our theorem can most likely be extended to negative integers $q\ne-1$; we do not pursue this task.

\section{Transcendence of Schur's $q$-Fibonacci sequence}
\label{s2}

The following result of Rosen from \cite{Ro20} gives a criterion for $\bt=(t_p)_{p\in\cP}\in\cA$ to be $\cA$-algebraic.

\begin{theorem}[{\cite[Theorem 1.1]{Ro20}}]
\label{thm:Ro}
Let $\bt=(t_p)_{p\in\cP}\in\cA$. The following conditions are equivalent.
\begin{itemize}
\item[(i)] The element $\bt$ is a finite algebraic number.
\item[(ii)] There exists a Galois extension $L/\mathbb Q$ and a map
$\phi\colon\Gal(L/\mathbb Q)\to L$ satisfying
\[
\phi(\sigma \tau\sigma^{-1})=\sigma(\phi(\tau)) \quad\text{for all}\; \sigma,\tau\in\Gal(L/\mathbb Q)
\]
such that
\[
(t_p)_{p\in\cP}=(\phi(\Frob_p)\bmod p)_{p\in\cP},
\]
where $\Frob_p$ is the Frobenius map of $L$ at the prime~$p$.
\end{itemize}
\end{theorem}

Before going to the details of our proof of Theorem \ref{thm:Ourtheorem} let us review the strategy of establishing the transcendence in \cite{AF24}.
The core of the argument there is the following important observation about the $q$-Fibonacci numbers.

\begin{theorem}[{\cite[Theorem 1.1]{AF24}}]
\label{thm:AF}
For $\alpha\in \mathbb Q^*\setminus\{1\}$ and $p\in\cP$ satisfying $\nu_p(\alpha)=\nu_p(\alpha-\nobreak1)=0$ and $\ord_p(\alpha)\not\equiv 0\bmod 5$, we have
\[
F_p(\alpha)\equiv F_{I_p(\alpha)+\left(\frac{\ord_p(\alpha)}{5}\right)}\bmod p.
\]
Here $I_p(\alpha)=(p-1)/\ord_p(\alpha)$ is the residual index of $\alpha$ modulo~$p$, while $\big(\frac{\cdot}{5}\big)$ is the Legendre symbol. 
\end{theorem}

This theorem gives the residue class of $F_p(\alpha)$ modulo $p$ for a positive proportion of all primes (say, all primes $p\equiv 2,3,4\bmod 5$ which are sufficiently large, so that they do not divide $ab(a-b)$ in the representation of $\alpha=a/b$ with coprime integers $a,b$). Using work of Moree from \cite{Mo06}, the authors of \cite{AF24} prove the transcendence of $\bF(\alpha)$ for $\alpha=q\in\mathbb Z_{>1}$ assuming additional hypotheses. 

\begin{remark}
Algebraic elements in $\cA$ as defined by Rosen form a subalgebra of $\cA$; this subalgebra can be viewed as a finite analogue of the algebraic closure $\overline{\mathbb Q}$ of~$\mathbb Q$ (see \cite[Section~4]{Ro20}).
There is an alternative notion of `na\"\i ve' $\cA$-algebraic number originated in \cite{AF24}: it is an element $(t_p)_{p\in\cP}$ for which there is a polynomial $Q(x)\in\mathbb Q[x]$ such that $(Q(t_p))_{p\in\cP}$ is zero in~$\cA$.
It is not hard to see that $\cA$-algebraic numbers of Rosen are automatically na\"\i ve $\cA$-algebraic numbers and that for detecting the na\"\i ve $\cA$-transcendence one can still use Proposition~\ref{pr1}:
Theorem~1.2 from \cite{AF24} covers both the na\"\i ve and Rosen's $\cA$-transcendence of~$\bF(q)$.
At the same time, the algebra of na\"\i ve $\cA$-algebraic numbers does not truly display an algebraic structure: for example, any element $(t_p)_{p\in\cP}$ with \emph{any} choice of $t_p\in\{\pm1\}$ (and there are uncountably many of those) is a root of $x^2-1$.
Our proof of Theorem~\ref{thm:Ourtheorem} makes use of Rosen's Theorem~\ref{thm:Ro}, so we do not discuss the na\"\i ve $\cA$-transcendence of~$\bF(q)$.
\end{remark}

\begin{proof}[Proof of Theorem~\textup{\ref{thm:Ourtheorem}}]
We assume for a contradiction  that $\bF(q)$ is $\cA$-algebraic. Thus, there exist data $L/\mathbb Q$ and $\phi\colon\Gal(L/\mathbb Q)\to L$ as in (ii) of Theorem \ref{thm:Ro}. Let 
\[
\{\alpha_1,\dots,\alpha_k\}=\phi(\Gal(L/\mathbb Q))
\] 
be the finite set of algebraic numbers given by the image of $\Gal(L/\mathbb Q)$ in~$L$.
It follows that, for all large primes $p$, there is some $i=i(p)\in \{1,\dots,k\}$ such that  
\[
F_p(q)\equiv \alpha_{i(p)}\bmod p.
\]
Write $q=q_0^{\lambda}$, where $\lambda\ge 1$ is a positive integer and $q_0>1$ is not a perfect power of any other integer. Let $r$ be a sufficiently large positive integer which is coprime to $5\lambda$ such that 
\begin{equation}
\label{eq:max}
F_{r-1}>\max\{\N_{L/\mathbb Q}(\alpha_i): 1\le i\le k\},
\end{equation}
where $\N$ stands for the norm.
Let $X$ be a large real number, and consider primes $p\in [X,2X]$ such that 
\[
p \equiv 1 \bmod r, \quad
p \equiv -1 \bmod {5\lambda}, \quad\text{and}\quad
\ord_p(q) \mid \frac{p-1}{r}.
\]
The Frobenius of the above primes form a conjugacy class in $\Gal(M/\mathbb Q)$, where $M=\mathbb Q(e^{2\pi i/(5\lambda r)}, \sqrt[r]{q_0})$. 
In particular, the number of such primes is 
\begin{equation}
\label{eq:large}
\gg \frac{X}{r^2\log X},
\end{equation}
where the multiplicative constant implied by the Vinogradov symbol depends on~$q$. 
Note that for such primes $\ord_p(q)=\ord_p(q_0)$, so that 
\[
\ord_p(q_0)=\ord_p(q)=\frac{p-1}{dr},
\]
where $d$ is some divisor of $(p-1)/r$.
Note that $\ord_p(q)$ is coprime to~$5$; it then follows that $I_p(q)=dr$.
Thus, 
\[
F_{I_p(q)+\left(\frac{\ord_p(q)}{5}\right)}=F_{dr+\eta} \quad\text{where}\; \eta=\left(\frac{\ord_p(q)}{5}\right)\in\{\pm 1\}.
\]
Therefore, $F_{dr+\eta}\equiv \alpha_i\bmod p$ with $i=i(p)$; in particular, 
\[
p\mid\N_{L/\mathbb Q}(F_{dr+\eta}-\alpha_i).
\]
Note that the norm is non-zero by condition~\eqref{eq:max}. Take
\[
(\alpha,\beta)=\bigg(\frac{1+{\sqrt{5}}}{2}, \frac{1-{\sqrt{5}}}{2}\bigg)
\]
to write 
\[
F_n=\frac{\alpha^n-\beta^n}{\sqrt{5}}.
\] 
Then
\begin{align*}
F_{dr+\eta}-\alpha_i
& = \frac{\alpha^{dr+\eta}-\beta^{dr+\eta}}{\sqrt{5}}-\alpha_i \\
& = \frac{\alpha^{-dr+\eta}}{\sqrt{5}}\big(\alpha^{2dr}-({\sqrt{5}}\alpha^{-\eta} \alpha_i)\alpha^{dr}+(-1)^{dr}\big) \\
& = \frac{\alpha^{-dr+\eta}}{\sqrt{5}}(\alpha^{dr}-\alpha_{i,\eta,\epsilon}')(\alpha^{dr}-\alpha_{i,\eta,\epsilon}''),
\end{align*}
where $\epsilon\in \{0,1\}$ is such that $dr\equiv \epsilon\bmod 2$ and $\alpha_{i,\eta,\epsilon}'$, $\alpha_{i,\eta,\epsilon}''$ are the roots of the polynomial
\[
x^2-(\sqrt{5}\alpha^{-\eta}\alpha_i)x+(-1)^{\epsilon}.
\]
In particular, the number $\alpha_{i,\eta,\epsilon}$ is at most quadratic over~$M$ (note that $M$ already contains $\sqrt{5}$). Varying $i\in\{1,\dots,k\}$, $\eta\in\{\pm 1\}$ and $\epsilon\in \{0,1\}$ and putting 
\begin{gather*}
N=M(\alpha_{i,\eta,\epsilon}: i\in \{1,\dots, k\},\; \eta  \in \{\pm 1\},\; \epsilon\in \{0,1\}),
\\
\{\beta_1,\dots,\beta_\ell\}=\{\alpha_{i,\eta,\epsilon}', \alpha_{i,\eta,\epsilon}'': 1\le i\le k,\;\eta \in \{\pm 1\}, \;\epsilon\in \{0,1\}\},
\end{gather*}
we get that for each such $p$ there exists a prime ideal $\frak p$ in $N$ sitting above $p$ such that 
$\frak p\mid \alpha^{dr}-\beta_j$
for some $j=1,\dots,\ell$. At the same time we also have that 
$p\mid q_0^{(p-1)/(dr)}-1$.
Thus, 
\[
\frak p\mid \gcd(\alpha^{dr}-\beta_j,q_0^{(p-1)/(dr)}-1).
\]
Set $(m,n)=(dr,(p-1)/(dr))$. Since $mn=p-1\le 2X$, we distinguish three situations.

\noindent
\textbf{Case 1}: $n\le\sqrt{X}/\log X$. 
Let $\cP_1$ be the set of primes in this category. For $p\in\cP_1$, we see that there exists $n\le \sqrt{X}/\log X$ such that 
$p\mid q_0^{n}-1$.
Therefore,
\[
\prod_{p\in \cP_1} p\le \prod_{n\le \sqrt{X}/\log X} (q_0^n-1)
<\exp\bigg((\log q_0)\sum_{n\le \sqrt{X}/ \log X} n\bigg)
=\exp\bigg(O\bigg(\frac{X}{(\log X)^2}\bigg)\bigg).
\]
Since $p\ge X$ for all $p\in \cP_1$, it follows that 
\[
\#\cP_1\ll \frac{X}{(\log X)^3}.
\]

\noindent
\textbf{Case 2}: $m\le \sqrt{X}/\log X$.
Let $\cP_2$ be the set of primes in this category. Then 
\[
p\mid\N_{N/\mathbb Q}(\alpha^m-\beta_j)
\]
for some $j=1,\dots,\ell$ and some $m\le\sqrt{X}/\log X$. Multiplying the above divisibility relations altogether we get
\begin{align*}
\prod_{p\in \cP_2} p
\Bigm| \prod_{\substack{1\le j\le \ell\\ t\le m\le\sqrt{X}/\log X}}\N_{N/\mathbb Q}(\alpha^m-\beta_j)
&\le \exp\bigg(O\bigg(\sum_{m\le \sqrt{X}/\log X} m\bigg)\bigg) \\
&=\exp\bigg(O\bigg(\frac{X}{(\log X)^2}\bigg)\bigg).
\end{align*}
As in Case~1 we deduce that that
\[
\#\cP_2\ll \frac{X}{(\log X)^3}.
\]

\noindent
\textbf{Case 3}: $n>\sqrt{X}/\log X$ and $m>\sqrt{X}/\log X$. 
Let $\cP_3$ be the set of such primes. It then follows that $n$ is a divisor of $p-1$ in the interval $[{\sqrt{X}}/\log X, 2{\sqrt{X}}\log X]$. By results of Ford \cite[Theorem 1\,(v) and Theorem~6]{Fo08} we obtain that 
\[
\#\cP_3\le \frac{X(\log\log X)^{O(1)}}{(\log X)^{1+\delta}},
\]
where
\[
\delta=1-\frac{1+\log\log 2}{\log 2}=0.086071\dots
\]
is the Erd\H os--Ford--Tenenbaum constant.

Summarising, the set of our primes is contained in $\cP_1\cup\cP_2\cup\cP_3$, whose total count is
\[
\le \#\cP_1+\#\cP_2+\#\cP_3
\ll \frac{X(\log\log X)^{O(1)}}{(\log X)^{1+\delta}}.
\]
On the other hand, the number of such primes is at least $X/(r^2\log X)$ by what we said at \eqref{eq:large}. We thus get that 
\[
\frac{X}{r^2\log X} \ll \frac{X(\log\log X)^{O(1)}}{(\log X)^{1+\delta}},
\]
which is false for $X>X(r)$ sufficiently large with respect to~$r$. This is the desired contradiction. 
\end{proof}

\section{Concluding remarks}
\label{s3}

Returning to Example~\ref{ex:sqrt2}, one can use the recursion $t_{n+4}+t_n=0$ to define the sequence $\big(2(\tfrac{2}{n})\big)_{n\ge3}$, which will then be equivalent to $\bs(1)$.
At the same time, the sequences $(s_n(1))_{n\ge3}$ and $\big(2(\tfrac{2}{n})\big)_{n\ge3}$ satisfy different recursions, over $\mathbb Z[n]$:
for the former one, we get the inhomogeneous recurrence equation
\[
s_{n+1}(1)-s_n(1)=\frac{(-1)^{n-1}C_{n-1}}{2^{2n-1}} \quad\text{for}\; n=1,2,\dots,
\]
which can be converted into a homogeneous one (because the inhomogenuity satisfies a first-order recursion).

Is there an algorithmic strategy to prove the equality in the ring~$\cA$?
In other words, assume that two rational-valued sequences $(t_n)_{n\gg1}$ and $(s_n)_{n\gg1}$ (eventually) satisfy some linear recurrence equations with coefficients in $\mathbb Z[n]$ (perhaps, with an extra condition that the degree of both boundary coefficients is also the maximal degree of all its coefficients, to guarantee that the characteristic polynomial does not degenerate).
Is it possible to check, from the recursions and initial data, whether $t_p\equiv s_p\bmod p$ for all primes $p\gg1$ or not?
Equivalently (by passing to the difference $t_p-s_p$), to decide when the $\cA$-projection of a solution to a recursion with polynomial coefficients is zero?

Of course, the square root example above is too simplistic.
One potential interesting application would be to the sequence
\[
t_n=(-1)^{n-1}\sum_{k=1}^{n-1}\frac{(-1)^k}{k}
\]
satisfying a simple recursion $t_{n+1}+t_n=1/n$ or, after homogenisation,
\[
nt_{n+1}+t_n-(n-1)t_{n-1}=0.
\]
We expect $(t_p\bmod p)_{p\in\cP}\in\cA$ to be irrational (and even transcendental), in particular non-zero;
the latter corresponds to $2^{p-1}\not\equiv1\bmod{p^2}$ happening for infinitely many primes $p$\,---\,a particular prediction from Campana's conjectures, which we do not know how to prove (though it is known to follow from the $abc$-conjecture \cite{Si88}).
The irrationality of $(t_p\bmod p)_{p\in\cP}$ is a finite analogue of the irrationality of $\log2$ (see \cite{KMS25}), and perhaps the first step to do before looking into the irrationality of $Z_{\cA}(k)$ for $k>1$~odd.

The similarity between this last example and $Z_{\cA}(3)$ discussed in Section~\ref{s1} is witnessed through Wolstenholme's theorem
\[
\sum_{k=1}^{p-1}\frac{1}{k}\equiv\frac13p^2B_{p-3}\bmod{p^3}
\quad\text{for primes}\; p>3;
\]
more generally\,---\,thanks to Wolstenholme and Glaisher \cite{Me11}\,---\,we have
\[
-\frac{(m-1)(m-2)}{2m}\,p^2B_{p-m}\equiv\sum_{k=1}^{p-1}\frac1{k^{m-2}}\bmod{p^3}
\]
and
\[
\frac{m-1}{m}\,pB_{p-m}\equiv\sum_{k=1}^{p-1}\frac1{k^{m-1}}\bmod{p^2}
\]
for $m>1$ odd and primes $p>m$.

A family of transcendental examples, very different from the one treated in our Theorem~\ref{thm:Ourtheorem}\,---\,though conditional, corresponds to the sequence of Frobenius traces $\ba=(a_p)_{p\in\cP}$ attached to an elliptic curve $E$ without complex multiplication. The Lang--Trotter conjecture predicts, in particular, that infinitely many values of $r\in\mathbb Z$ are attained by the trace, $a_p=r$, so that the $\cA$-transcendence follows from Proposition~\ref{pr1}.
At the same time, the $\cA$-irrationality of $\ba$ is unconditional:
a theorem of Elkies \cite{El87} shows there are infinitely many primes $p$ with $a_p=0$;
at the same time the number of primes $p\le x$ for which $a_p=0$ is $O(x^{3/4})$\,---\,this can be derived by using a result of Kaneko~\cite{Ka89},
so that there are infinitely many primes $p$ with $a_p\ne0$, hence $a_p\not\equiv0\bmod p$.
The irrationality also holds for CM elliptic curves thanks to Deuring's theorem \cite{De41}.
Similar (conditional) conclusions can be made for the eigenvalues $(a_p)_{p\in\cP}$ of a Hecke eigenform of an arbitrary weight.

\begin{acknowledgements}
We are thankful to Gunther Cornelissen, Andrew Granville, Masanobu Kaneko, Toshiki Matsusaka, Pieter Moree and Igor Shparlinski for their comments on preliminary drafts of this note.
Ideas for this project were cemented at the conference on \emph{Asymptotic Counting and $L$-Functions} in the Max Planck Institute for Mathematics (Bonn, Germany) in May 2025.
We thank the participants of the event and staff of the institute for a unique inspiring atmosphere of the week.

F.L.\ was partly supported by the 2024 ERC Synergy Project DynAMiCs.
\end{acknowledgements}

%==================================================

\end{document}